\documentclass{amsart}

\usepackage{hyperref}
\hypersetup{colorlinks,linkcolor={red},citecolor={blue},urlcolor={blue}}

\newtheorem{theorem}{Theorem}[section]
\newtheorem{lemma}[theorem]{Lemma}

\newtheorem{corollary}[theorem]{Corollary}

\theoremstyle{definition}
\newtheorem{definition}[theorem]{Definition}

\theoremstyle{remark}
\newtheorem{remark}[theorem]{Remark}
\newtheorem{questions}[theorem]{Questions}

\numberwithin{equation}{section}

\newcommand{\CC}{\mathbb C}
\newcommand{\HH}{\mathbb H}
\newcommand{\NN}{\mathbb N}
\newcommand{\cH}{\mathcal H} 
\newcommand{\cP}{\mathcal P}
\newcommand{\cD}{\mathcal D}
\newcommand{\QQ}{\mathbb Q}
\newcommand{\RR}{\mathbb R}
\newcommand{\ZZ}{\mathbb Z}
\newcommand{\PP}{\mathbb P}
\newcommand{\SL}{\mathop{\mathrm {SL}}\nolimits}

\newcommand{\Mp}{\mathop{\mathrm {Mp}}\nolimits}
\newcommand{\rank}{\mathop{\mathrm {rank}}\nolimits}
\def\Orth{\operatorname{O}}

\def\Borch{\operatorname{Borch}}
\def\div{\operatorname{div}}

\begin{document}

\title{Reflective modular forms: A Jacobi forms approach}

\author{Haowu Wang}

\address{Laboratoire Paul Painlev\'{e}, Universit\'{e} de Lille, 59655 Villeneuve d'Ascq Cedex, France}

\email{haowu.wang@univ-lille.fr}


\date{\today}


\begin{abstract}
We give an explicit formula to express the weight of $2$-reflective modular forms. We prove that there is no $2$-reflective lattice of signature $(2,n)$ when $n\geq 15$ and $n\neq 19$ except the even unimodular lattices of signature $(2,18)$ and $(2,26)$. As applications, we give a simple proof of Looijenga's theorem that the lattice $2U\oplus 2E_8(-1)\oplus\langle -2n\rangle$ is not $2$-reflective if $n>1$. We also classify reflective modular forms on lattices of large rank and the modular forms with the simplest reflective divisors.
\end{abstract}

\maketitle

\section{Introduction}
Let $M$ be an even lattice of signature $(2,n)$. A non-constant holomorphic modular form for $M$ is called reflective if the support of its divisor is  contained in the union of quadratic divisors determined by reflective vectors of $M$. This type of modular forms first appeared in the works of Borcherds \cite{Bo1,Bo2} and Gritsenko-Nikulin \cite{GN}. They have many applications in various related topics, such as the classification of Lorentzian Kac-Moody algebras \cite{Bar03, GN1, GN1v1, GN2, GN3,S1,S2}, search of hyperbolic reflection groups \cite{Bar03, Bo3} and the theory of moduli spaces \cite{GH1,GHS,L,Ma2}.

Reflective modular forms seem to be exceptional and very rare. The classification of reflective modular forms has been widely studied by several mathematicians. In 1998, Gritsenko and Nikulin first conjectured that the number of lattices having reflective modular forms is finite \cite{GN1} and gave a complete classification for $n=3$ \cite{GN1v1,GN2}. Later Scheithauer gave a complete classification of reflective modular forms of singular weight on lattices of prime level \cite{S1, S2, S3, S4}. Looijenga \cite{L} proved one part of the arithmetic mirror symmetry conjecture formulated in \cite{GN1v2}, which might give a new approach to classify reflective modular forms. Recent work of Ma \cite{Ma2} showed that there are only finitely many lattices of signature $(2,n)$ which carry a strongly reflective modular form that vanishes at order one along the reflective divisors when $n\geq 4$.

The aim of this paper is to investigate 2-reflective modular forms which are the most basic class of reflective modular forms.  A non-constant holomorphic modular form for $M$ is called $2$-reflective if the support of its zero divisor is contained in the Heegner divisor defined by the $(-2)$-vectors in $M$. A lattice $M$ is called $2$-reflective if it admits a $2$-reflective modular form. In 2017, Ma \cite{Ma1} showed that there are only finitely many $2$-reflective lattices of signature $(2,n)$ with $n\geq 7$ and there is no $2$-reflective lattice when $n\geq 26$ except the even unimodular lattice ${II}_{2,26}$ of signature $(2,26)$. In this paper, we prove the following main theorem.

\begin{theorem}[Theorem \ref{th2'}]
Let $M$ be a $2$-reflective lattice of signature $(2,n)$ with $n\geq 15$. Then $n=19$ or $M$ is isomorphic to the unique even unimodular lattice of signature $(2,18)$ or $(2,26)$.
\end{theorem}

Our approach is based on the Gritsenko-Nikulin representation of Borcherds products in terms of Jacobi forms \cite{G3, GN1v1}. When $M$ contains two integral hyperbolic planes (i.e. $M=2U\oplus L(-1)$), every $2$-reflective modular form can be represented as a Borcherds product \cite{Br2}. By means of the isomorphism between vector-valued modular forms and Jacobi forms, there exists a weakly holomorphic Jacobi form of weight $0$ and index $L$ such that its Borcherds product gives the above $2$-reflective modular form. Then the identity (Lemma \ref{L2}) related to $q^0$-term of Jacobi forms of weight $0$ yields a formula expressing the weight of $2$-reflective modular forms (Theorem \ref{th1}). Furthermore, we can construct holomorphic Jacobi forms of certain small weights from the above Jacobi form of weight $0$ by using the weight raising differential operators (Lemma \ref{L1}). The existence of such Jacobi forms implies the non-existence of $2$-reflective modular forms with respect to lattices of large rank. Our main theorem gives us a necessary condition for a lattice of signature $(2,19)$ being $2$-reflective (Theorem \ref{th2}). From this, we deduce that the interesting lattices $T_n=2U\oplus 2E_8(-1)\oplus \langle -2n\rangle$  are not $2$-reflective for $n\geq 2$ (Theorem \ref{th3}). Note that this question has been investigated in \cite{GN2, N} and the above result was first proved by Looijenga in \cite{L}.

The above arguments can also be used to classify reflective modular forms. In section \ref{Sec:4}, we study reflective modular forms on lattices of prime level and show that such modular forms do not exist when the rank of lattice is large enough, this generalizes a result of Scheithauer \cite[Theorem 6.5]{S4}. More precisely, we show

\begin{theorem}[Theorem \ref{coro}]
Let $M$ be a reflective lattice of signature $(2,n)$ and of prime level $p$.
\begin{enumerate}
\item If $p=2$ and $n>18$ and $n\neq 22$, then $M$ is isomorphic to $II_{2,26}(2)$.
\item If $p=3$ and $n>14$ and $n\neq 20$, then $M$ is isomorphic to $II_{2,18}(3)$ or $II_{2,26}(3)$.
\item If $p>3$ and $n> 10+ 24/(p+1)$, then $M$ is isomorphic to $II_{2,18}(p)$ or $II_{2,26}(p)$.
\end{enumerate}
\end{theorem}

In \cite{Ma1}, Ma showed that there is no reflective lattice of signature $(2,n)$ with $n\geq 26$ containing $2U$ except the even unimodular lattice of signature $(2,26)$. As an extension of Ma's result,  we prove the following result. 

\begin{theorem}[Theorem \ref{prop01}]
There is no reflective lattice of signature $(2,n)$ with $23\leq n\leq 25$.
\end{theorem}

Our approach can be applied to some other questions. For instance, it provides us a straightforward way to classify the modular forms with the simplest reflective divisors, i.e. the dd-modular forms (see Section \ref{Sec:5}). This gives a generalization of the main results in \cite{CG1, GH} (see Theorem \ref{th4}).

The paper is organized as follows. In section \ref{Sec:2} we introduce briefly Jacobi forms and differential operators. In section \ref{Sec:3} we define 2-reflective modular forms and prove our first main theorem. An application is also presented. In section \ref{Sec:4} we classify reflective modular forms.  Section \ref{Sec:5} is devoted to the classification of dd-modular forms.

\section{Preliminaries: Jacobi forms}\label{Sec:2}
In this section, some standard facts about Jacobi forms are recalled.  We refer to \cite{CG2,G} for more details. From now on, $L$ always denotes an even positive-definite lattice with bilinear form $(\cdot, \cdot)$ and dual lattice $L^\vee$. The rank of $L$ is denoted as $\rank(L)$. We define the Jacobi forms in the following way.

\begin{definition}
Let $\varphi : \HH \times (L \otimes \CC) \rightarrow \CC$ be a holomorphic function and $k\in\ZZ$. If $\varphi$ satisfies the functional equations
\begin{align}
\varphi \left( \frac{a\tau +b}{c\tau + d},\frac{\mathfrak{z}}{c\tau + d} \right)& = (c\tau + d)^k e^{i \pi \frac{c(\mathfrak{z},\mathfrak{z})}{c \tau + d}} \varphi ( \tau, \mathfrak{z} ),\\
\varphi (\tau, \mathfrak{z}+ x \tau + y)&= e^{-i \pi ( (x,x)\tau +2(x,\mathfrak{z}) )} \varphi ( \tau, \mathfrak{z} )
\end{align}
for any $\left( \begin{array}{cc}
a & b \\ 
c & d
\end{array} \right)   \in \SL_2(\ZZ)$ and any $x,y \in L$
and $\varphi$ admits a Fourier expansion as 
\begin{equation}\label{JFFC}
\varphi ( \tau, \mathfrak{z} )= \sum_{n\geq n_0 }\sum_{l\in L^\vee}f(n,l)q^n\zeta^l
\end{equation}
where $n_0\in \ZZ$, $q=e^{2\pi i \tau}$ and $\zeta^l=e^{2\pi i (l,\mathfrak{z})}$, then $\varphi$ is called a weakly holomorphic Jacobi form of weight $k$ and index $L$. If $\varphi$ further satisfies the condition
$$ f(n,l) \neq 0 \Longrightarrow n \geq 0 $$
then $\varphi$ is called a weak Jacobi form.  If $\varphi$ further satisfies the stronger condition
$$ f(n,l) \neq 0 \Longrightarrow 2n - (l,l) \geq 0 $$
then $\varphi$ is called a holomorphic Jacobi form.
We denote by $J^{w.h.}_{k,L}$ (resp. $J^{w}_{k,L}$, $J_{k,L}$) the vector space of weakly holomorphic Jacobi forms (resp. weak Jacobi forms, holomorphic Jacobi forms) of weight $k$ and index $L$.
\end{definition}

Similarly, we can also define Jacobi forms with character. For more details, we refer to \cite{CG2, G, GN1v1}.

The Fourier coefficient $f(n,l)$ depends only on the number $2n-(l,l)$ and the class of $l$ modulo $L$. The fact implies that for fixed $n$ the sum $\sum_{l\in L^\vee}f(n,l)q^n\zeta^l$ in (\ref{JFFC}) is finite.
The number $2n-(l,l)$ is called the hyperbolic norm of Fourier coefficient $f(n,l)$. The Fourier coefficients $f(n,l)$ with negative hyperbolic norm are called singular Fourier coefficients, which play a crucial role in the theory of Borcherds products. By definition, a weakly holomorphic Jacobi form without singular Fourier coefficient is a holomorphic Jacobi form. 

We next explain the isomorphism between vector-valued modular forms and Jacobi forms. More details can be found in  \cite[Section 2]{CG2} and \cite[Lemma 2.3]{G}. Let $U$ be a hyperbolic plane, i.e. $U=\ZZ e\oplus\ZZ f$ with $(e,f)=1$, $(e,e)=(f,f)=0$. Then the lattice $2U\oplus L(-1)$ is an even lattice of signature $(2,\rank(L)+2)$, whose discriminant group is isomorphic to $D(L)=L^\vee/L$.  Let $\{\textbf{e}_\gamma: \gamma \in D(L)\}$ be the basis of the group ring $\CC[D(L)]$. Let $\Mp_2(\ZZ)$ be the metaplectic group which is a double cover of $\SL_2(\ZZ)$. We denote the Weil representation of $\Mp_2(\ZZ)$ on $\CC[D(L)]$ by $\rho_{D(L)}$ (see \cite[Section 1.1]{Br1}). Let $F$ be a nearly holomorphic vector-valued modular form for $\rho_{D(L)}$ of weight $k$ with  Fourier expansion  of the form
\begin{align*}
F(\tau)&=\sum_{\gamma\in D(L)}\sum_{n\in\ZZ-\frac{(\gamma,\gamma)}{2}} c(\gamma,n)q^n \textbf{e}_\gamma\\
&  =\sum_{\gamma\in D(L)} F_{\gamma}(\tau)\textbf{e}_\gamma.
\end{align*}
Here, the sum 
$$
\sum_{\gamma\in D(L)}\sum_{\substack{ n\in\ZZ-\frac{(\gamma,\gamma)}{2}\\ n<0}} c(\gamma,n)q^n \textbf{e}_\gamma
$$
is called the principal part of $F$.
We define the theta-function for the lattice $L$ as
\begin{equation}
\Theta_{\gamma}^{L}(\tau,\mathfrak{z})=\sum_{l \in \gamma +L}\exp\left(\pi i(l,l) \tau + 2\pi i(l,\mathfrak{z}) \right), \quad \gamma\in D(L).
\end{equation}
Then the function
\begin{equation}
\sum_{\gamma\in D(L)} F_{\gamma}(\tau)\Theta_{\gamma}^{L}(\tau,\mathfrak{z})
\end{equation}
is a weakly holomorphic Jacobi form of weight $k+\frac{1}{2}\rank (L)$ and index $L$. The principal part of $F$ corresponds to the singular Fourier coefficients of the above Jacobi form.
The map 
\begin{align*}
M_{k,\rho_{D(L)}} &\to J_{k+\frac{1}{2}\rank(L),L}\\
F(\tau) &\mapsto \sum_{\gamma\in D(L)} F_{\gamma}(\tau)\Theta_{\gamma}^{L}(\tau,\mathfrak{z})
\end{align*}
is an isomorphism sending holomorphic vector-valued modular forms to holomorphic Jacobi forms. Since the weights of holomorphic vector-valued modular forms are non-negative, the space $J_{k,L}$ of holomorphic Jacobi forms of weight $k$ and index $L$ is trivial if $k<\rank(L)/2$. The minimum possible weight $k=\rank(L)/2$ is called the singular weight.

The Borcherds product is a map from vector-valued modular forms for the Weil representation of $\Mp_2(\ZZ)$ to modular forms on orthogonal groups (see \cite{Bo1, Bo2}). By means of the isomorphism between vector-valued modular forms for the Weil representation and Jacobi forms, Gritsenko and Nikulin (\cite[Theorem 3.1]{G3} or \cite[Theorem 2.1]{GN}) gave a variant of Borcherds products, which lifts weakly holomorphic Jacobi forms of weight $0$ to modular forms on orthogonal groups. We denote the Borcherds product of a weakly holomorphic Jacobi form $\phi$ of weight $0$ by $\Borch(\phi)$.

We now recall the following weight raising differential operator which will be used later. Such technique can also be found in \cite{CK} for the general case or in \cite{EZ} for classical Jacobi forms.

\begin{lemma}\label{L1}
Let $\psi(\tau,\mathfrak{z})=\sum a(n,l)q^n\zeta^l$ be a weakly holomorphic Jacobi form of weight $k$ and index $L$. Then $H_k(\psi)$ is a weakly holomorphic Jacobi form of weight $k+2$ and index $L$, where 
\begin{align}
H_k(\psi)&=H(\psi)+(2k-\rank(L))G_2\psi,\\
H(\psi)(\tau,\mathfrak{z})&=\frac{1}{2}\sum_{n\in \ZZ}\sum_{l\in L^\vee} \left[2n-(l,l) \right]a(n,l)q^n\zeta^l,
\end{align}
and $G_2(\tau)=-\frac{1}{24}+\sum_{n\geq 1}\sigma(n)q^n$ is the Eisenstein series of weight $2$ on $\SL_2(\ZZ)$.
\end{lemma}
\begin{proof}
Let $\{\alpha_1,...,\alpha_n\}$ be a basis of $L$ and $\{\alpha_1^*,...,\alpha_n^*\}$ be its dual basis. We write $\mathfrak{z}=\sum_{i=1}^nz_i\alpha_i \in L\otimes \CC$, $z_i\in\CC$. We define $\frac{\partial}{\partial\mathfrak{z}}=\sum_{i=1}^n\alpha_i^* \frac{\partial}{\partial z_i}$. Then we have that
$$
\left(\frac{\partial}{\partial\mathfrak{z}},\frac{\partial}{\partial\mathfrak{z}}\right)e^{2\pi i(l,\mathfrak{z})}=-4\pi^2(l,l)e^{2\pi i(l,\mathfrak{z})}
$$
and the operator $H(\cdot)$ is equal to the heat operator
$$
H=\frac{1}{2\pi i}\frac{\partial}{\partial\tau}+\frac{1}{8\pi^2}\left(\frac{\partial}{\partial\mathfrak{z}},\frac{\partial}{\partial\mathfrak{z}}\right).
$$
By formulas (3.5) and (3.7) in \cite[Lemma 3.3]{CK}, the transformations of the function $H(\psi)$ with respect to the actions of $\SL_2(\ZZ)$ and the Heisenberg group of $L$ are known. From these transformations, we see that $H(\psi)$ does not transform like a Jacobi form, so we make an automorphic correction by considering the quasi-modular Eisenstein series $G_2$ of weight $2$. By direct calculations, we can show that $H_k(\psi)$ is invariant under $\SL_2(\ZZ)$ and the Heisenberg group. Therefore, it is a weakly holomorphic Jacobi form of weight $k+2$ and index $L$.
\end{proof}

The next lemma plays a key role in our discussions on the weight of $2$-reflective modular forms and in the classification of dd-modular forms. It is a particular case of \cite[Proposition 2.2]{G3}. 

\begin{lemma}\label{L2}
Assume that $\phi$ is a weakly holomorphic Jacobi form of weight $0$ and index $L$ with the Fourier expansion
$$\phi(\tau,\mathfrak{z})=aq^{-1}+\sum_{l\in L^\vee}c(0,l)\zeta^l+O(q).$$ Then we have the following identity
$$\sum_{l\in L^\vee}c(0,l)-\frac{12}{\rank(L)} \sum_{l\in L^\vee}c(0,l)(l,l)-24a=0. $$
\end{lemma}
\begin{proof}
From Lemma \ref{L1}, it follows that $H_{0}(\phi)$ is a weakly holomorphic Jacobi form of weight $2$. Therefore $H_{0}(\phi)(\tau,0)$ is a nearly holomorphic modular form of weight $2$ for the full modular group $\SL_2(\ZZ)$. By \cite[Lemma 9.2]{Bo1},  $H_{0}(\phi)(\tau,0)$ has zero constant term, which establishes the desired formula.
\end{proof}

\section{2-reflective modular forms}\label{Sec:3}
In this section we study 2-reflective modular forms. We prove our first main theorem and give an application of our results.
\subsection{Non-existence of 2-reflective modular forms}
Let $M$ be an even integral lattice of signature $(2,n)$, $n\geq 3$, and let
\begin{equation}
\cD(M)=\{[\omega] \in  \PP(M\otimes \CC):  (\omega, \omega)=0, (\omega,\bar{\omega}) > 0\}^{+}
\end{equation}
be the associated Hermitian symmetric domain of type IV (here $+$ denotes one of its two connected components). Let us denote the index $2$ subgroup of the orthogonal group $\Orth(M)$ preserving $\cD(M)$ by $\Orth^+ (M)$.  We define
\begin{equation}
\cH =\bigcup_{\substack{l \in M\\ (l,l)=-2}} l^{\perp}\cap \cD(M) 
\end{equation}
as the Heegner divisor of $\cD(M)$ generated by the $(-2)$-vectors in $M$.

\begin{definition}
Let $F$ be a non-constant holomorphic modular form on $\cD(M)$ with respect to a finite index subgroup $\Gamma < \Orth^+ (M)$ and a character (of finite order) $\chi : \Gamma \to \CC$. The function $F$ is called $2$-reflective if the support of its zero divisor is contained in $\cH$. A lattice $M$ is called $2$-reflective if it admits a $2$-reflective modular form.
\end{definition}

Following \cite{Ma1}, we consider the decomposition (\ref{F1}) of the $(-2)$-Heegner divisor $\cH$. Let $A_M=M^\vee / M$ be the discriminant group of $M$. We denote the most important subgroup of $\Orth^+ (M)$ acting trivially on $A_M$ by $\widetilde{\Orth}^+(M)$. The invariants of $\widetilde{\Orth}^+(M)$-orbit of a primitive vector $l\in M$ are its norm $(l,l)$ and its image $l/ \div(l)\in A_M$ where $\div (l)$ is the positive integer generating the ideal $(l,M)$. From this point of view, we choose the following notations.  For $\lambda \in A_M$ and $m\in \QQ$ we put
$$ \cH(\lambda,m)=\bigcup_{\substack{l \in M+\lambda \\ (l,l)=2m}} l^{\perp}\cap \cD(M) $$
as the Heegner divisor of discriminant $(\lambda,m)$. In particular, $\cH=\cH (0,-1)$. Let $\pi_M\subset A_M$ be the subset of elements of order $2$ and norm $-1/2$. For each $\mu \in \pi_M$ we abbreviate $\cH(\mu,-1/4)$ by $\cH_\mu$. We also set 
$$\cH_0 = \bigcup_{\substack{l \in M, (l,l)=-2\\ \div (l)=1}} l^{\perp}\cap \cD(M). $$
Then we have the following decomposition
\begin{equation}\label{F1}
\cH= \cH_0+ \sum_{\mu \in \pi_M} \cH_\mu.
\end{equation}

By \cite[Lemma 2.2]{Ma1}, if $M$ admits a 2-reflective modular form with respect to some $\Gamma_0< \Orth^+(M)$, then $M$ also has a 2-reflective modular form with respect to any other finite index subgroup $\Gamma <\Orth^+(M)$. Thus, the lattice $M$ is 2-reflective if and only if it admits a 2-reflective modular form with respect to $\widetilde{\Orth}^+(M)$. Throughout this section,  we only consider 2-reflective modular forms with respect to $\widetilde{\Orth}^+(M)$.

Next, we assume that the lattice $M$ contains $2U$ and $M=2U\oplus L(-1)$, where $L$ is a positive-definite even lattice and $U$ is a hyperbolic plane. In this case, each $\cH_*$ is an $\widetilde{\Orth}^+(M)$-orbit of a single quadratic divisor $l^{\perp}\cap \cD(M)$ and it is irreducible. We can write each element of $\pi_M$ in the form $\mu =(0,n_\mu,\mu_0/2,1,0)$, where $n_\mu \in \ZZ$, $\mu_0 \in L$ and $2n_\mu-\frac{1}{4}(\mu_0,\mu_0)=-\frac{1}{2}$.  If $M$ admits a $2$-reflective modular form $F$ of weight $k$, then its divisor can be written as 
\begin{equation}\label{d}
\begin{split}
\div (F)&= \beta_0 \cH_0 + \sum_{\mu \in \pi_M} \beta_\mu \cH_\mu \\
         &=\beta_0 \cH + \sum_{\mu \in \pi_M} (\beta_\mu-\beta_0) \cH_\mu,
\end{split}
\end{equation}
where $\beta_*$ are non-negative integers. By \cite[Theorem 5.12]{Br1} or \cite[Theorem 1.2]{Br2}, there exists a nearly holomorphic vector-valued modular form $f$ of weight $-\rank(L)/2$ with respect to the Weil representation $\rho_M$ of $\Mp_2(\ZZ)$ on the group ring $\CC[A_M]$ with principal part
$$\beta_0 q^{-1} \textbf{e}_0+ \sum_{\mu \in \pi_M} (\beta_\mu-\beta_0)q^{-1/4}\textbf{e}_\mu,$$
such that $F$ is the Borcherds product of $f$. In view of the isomorphism between vector-valued modular forms and Jacobi forms, there exists a weakly holomorphic Jacobi form $\phi_L$ of weight $0$ and index $L$ with singular Fourier coefficients of the form (see \cite{CG2})
\begin{equation}\label{F2}
sing(\phi_L)= \beta_0 \sum_{r\in L}q^{(r,r)/2-1}\zeta^r+ \sum_{\mu \in \pi_M}(\beta_\mu - \beta_0)\sum_{s\in L+ \mu_0/2}q^{(s,s)/2-1/4}\zeta^s
\end{equation}
where $\zeta^l=e^{2\pi i (l,\mathfrak{z})}$.
Thus, we have
\begin{equation}\label{F3}
\phi_L(\tau,\mathfrak{z})= \beta_0 q^{-1}+\beta_0 \sum_{r\in R(L)}\zeta^r +2k+ \sum_{\substack{u\in \pi_M}} (\beta_\mu -\beta_0) \sum_{s\in R_\mu (L)}\zeta^s + O(q)
\end{equation}
here and subsequently, $R(L)$ denotes the set of 2-roots in $L$ and 
\begin{equation}\label{F4}
R_\mu (L)=\{ s\in L^\vee : 2s\in R(L), s-\mu_0/2 \in L \}.
\end{equation}

\smallskip
 
With the help of equation (\ref{F3}) and Lemma \ref{L2}, we get the following theorem. 
\begin{theorem}\label{th1}
Let $L$ be a positive-definite even lattice and $M=2U\oplus L(-1)$. Suppose that $F$ is a $2$-reflective modular form of weight $k$ with divisor of the form \eqref{d}.
Then the weight $k$ of $F$ is given by the following formula
\begin{equation}\label{F5}
\begin{split}
k=&\beta_0\left[ 12+ \lvert R(L) \rvert \left(\frac{12}{\rank(L)}-\frac{1}{2} \right) \right]\\
&+\left(\frac{3}{\rank(L)}-\frac{1}{2} \right)\sum_{\mu \in \pi_M} (\beta_\mu - \beta_0 )\lvert R_\mu (L) \rvert.
\end{split}
\end{equation}
\end{theorem}

\begin{remark}\label{r3.3}
From the above theorem, we have
\begin{enumerate}
\item If $R(L)$ is empty, then the weight of $2$-reflective modular form is $12\beta_0$.
\item When $\rank (L) \geq 6$, there is no $2$-reflective modular form with $\beta_0=0$. This fact can also be proved by Riemann--Roch theorem as the proof of \cite[Proposition 6.1]{S4}.
\item When $\rank (L) = 6$, the weight of $2$-reflective modular form is $\beta_0(12+\frac{3}{2}\lvert R(L) \rvert)$ and the modular form is not of singular weight.
\end{enumerate}
\end{remark}

We next study modular forms with  complete $2$-divisor (i.e. $\div(F)=\cH$), which are the simplest 2-reflective modular forms.
\begin{theorem}\label{th}
If there exists a modular form with complete $2$-divisor for $M=2U\oplus L(-1)$, then either $\rank(L)\leq 8$, or $L$ is a unimodular lattice of rank $16$ or $24$. Moreover, the weight of the corresponding modular form is 
\begin{equation}
k=12+ \lvert R(L) \rvert \left(\frac{12}{\rank(L)}-\frac{1}{2} \right).
\end{equation}
\end{theorem}

\begin{proof}
Firstly, the above formula is a direct result of Theorem \ref{th1}.
Let $F$ be a modular form with complete $2$-divisor. Then there exists a weakly holomorphic Jacobi form of weight $0$ and index $L$ such that
$$ \phi(\tau,\mathfrak{z})=q^{-1}+\sum_{r\in R(L)}\zeta^r +2k + O(q)$$
whose singular Fourier coefficients are (see formula (\ref{F2}))
$$ sing (\phi)=\sum_{n\geq-1} \sum_{\substack{l\in L\\ (l,l)=2n+2}}q^n \zeta^l,$$
with $\Borch(\phi)=F$. By \cite{Ma1}, it is known that $\rank(L)<24$ or $L$ is a unimodular lattice of rank $24$. Let us assume that $\rank(L)\leq 23$ and let us construct two special Jacobi forms by using the differential operators introduced in Lemma \ref{L1}. For simplicity, we set $R=\lvert R(L) \rvert$ and $n_0=\rank(L)$.
\begin{align*}
f_2(\tau,\mathfrak{z})&=\frac{24}{n_0-24}H_0(\phi)(\tau,\mathfrak{z})\\
&=q^{-1}+\sum_{r\in R(L)}\zeta^r-R +O(q) \in J_{2,L}^{w.h.}\\
f_4(\tau,\mathfrak{z})&=\frac{24}{n_0-28}H_2(f_2)(\tau,\mathfrak{z})\\
&=q^{-1}+\sum_{r\in R(L)}\zeta^r-\frac{(R+24)(n_0-4)}{n_0-28}  +O(q) \in J_{4,L}^{w.h.}
\end{align*}
Let $E_4$ and $E_6$ denote the Eisenstein series on $\SL_2(\ZZ)$ of weight $4$ and $6$, respectively.
Then we can check that
\begin{align*}
g(\tau,\mathfrak{z})&=\frac{n_0-28}{48}\left[E_4(\tau)\phi(\tau,\mathfrak{z})-f_4(\tau,\mathfrak{z})\right]\\
&=R\left(1-\frac{14}{n_0} \right)+6(n_0-26)+ O(q)\in J_{4,L}
\end{align*}
and
\begin{align*}
h(\tau,\mathfrak{z})&=E_6(\tau)\phi(\tau,\mathfrak{z})-E_4(\tau)f_2(\tau,\mathfrak{z})\\
&=\frac{24R}{n_0}-720+ O(q) \in J_{6,L}
\end{align*}
are holomorphic Jacobi forms of weight $4$ and $6$, respectively. In fact, the singular Fourier coefficients are stable under the actions of the differential operators, so the singular Fourier coefficients of $f_2$ and $f_4$ come from $sing(\phi)$. In order to check $g$ and $h$ are holomorphic Jacobi forms, we only need to check that $g$ and $h$ have no singular Fourier coefficient i.e. the singular part $sing(\phi)$ has been cancelled by the above combinations of $\phi$, $f_2$ and $f_4$.

Since the singular weight of holomorphic Jacobi form of index $L$ is $\frac{n_0}{2}$, we deduce that  $g=0$ if $n_0>8$ and $h=0$ if $n_0>12$. By direct calculations, we have
\begin{itemize}
\item when $R=0$, $g \neq 0$ if $n_0<24$.
\item when $R>0$, $g \neq 0$ if $n_0\leq 14$. 
\item when $n_0\geq 15$,  $g = 0$ and $h=0$ if and only if $n_0=16$ and $R=480$.
\end{itemize}

When $n_0=16$, from $h=0$, it follows that the Fourier coefficients of $\phi$ satisfy:
$c(n,l)=0$ if $2n-(l,l)=0$ and $l\not\in L$. Otherwise, there exists a Fourier coefficient $c(n,l)\neq 0$ with  $2n-(l,l)=0$ and $l\not\in L$. Assume that $c(n,l)$ is such Fourier coefficient with the smallest $n$. Then the coefficient of $q^n\zeta^l$ in $E_6\phi$ is $c(n,l)$ and the coefficient of $q^n\zeta^l$ in $E_4f_2$ is $-2c(n,l)$. Thus, the coefficient of $q^n\zeta^l$ in $h$ is not zero and then $h\neq 0$, which leads to a contradiction. 
Therefore, the following holomorphic Jacobi form of singular weight 8
$$E_8\phi-E_6f_2=1728+\sum_{\substack{n>0,l\in L^\vee\\ 2n=(l,l)}}a(n,l)q^n\zeta^l \in J_{8,L}$$
satisfies the same condition: $a(n,l)=0$ if $2n-(l,l)=0$ and $l\not\in L$. We then obtain 
$$E_8\phi-E_6f_2=1728\sum_{l\in L}q^{\frac{(l,l)}{2}}\zeta^l$$
and $L$ has to be unimodular. The proof is completed.
\end{proof}

\begin{remark}
It is worth pointing out that there exist lattices $L$ which admit a modular form with complete 2-divisor when $1\leq \rank(L)\leq 8$ (see \cite{GN3}). For example, the Igusa cusp form $\chi_{35}$ is a modular form with complete 2-divisor.
\end{remark}

We are going to generalize our method to prove the non-existence of $2$-reflective modular forms in higher dimensions.

\begin{theorem}\label{th2}
Suppose that $M=2U\oplus L(-1)$ is a $2$-reflective lattice satisfying $\rank(L)\geq 13$. Then either $\rank(L)=17$, or $L$ is a unimodular lattice of rank $16$ or $24$. Furthermore, when $\rank(L)=17$, the weight of the corresponding $2$-reflective modular form is $75\beta_0$, where $\beta_0$ is the multiplicity of the divisor $\cH_0$.
\end{theorem}

\begin{proof}
If $M$ has a $2$-reflective modular form $F$ of weight $k$ with divisor of the form (\ref{d}), then there exists a weakly holomorphic Jacobi form $\phi$ of weight $0$ and index $L$ with singular Fourier coefficients of the form (\ref{F2}). Let us assume that $\rank(L)\leq 23$. We next construct a holomorphic Jacobi form of weight $6$ from $\phi$. We write $\phi=S_1 + d + S_2 + \cdots$, where $S_1$ and $S_{2}$ are the first and second terms in (\ref{F2}), respectively, and $d=2k$. It is clear that $\phi - S_1 -S_2$ does not have the term with negative hyperbolic norm. We can construct ($n_0=\rank(L)$)
\begin{align*}
f_2&=\frac{24}{n_0-24}H_0(\phi)=S_1+ d_1 +c_1 S_2 +\cdots \in J_{2,L}^{w.h.},\\
f_4&=\frac{24}{n_0-28}H_2(f_2)=S_1+ d_2 +c_2 S_2 +\cdots \in J_{4,L}^{w.h.},\\
f_6&=\frac{24}{n_0-32}H_4(f_4)=S_1+ d_3 +c_3 S_2 +\cdots \in J_{6,L}^{w.h.},
\end{align*}
where
\begin{align*}
&d_1=\frac{n_0(d-24\beta_0)}{n_0-24},& &c_1=\frac{n_0-6}{n_0-24},&\\
&d_2=\frac{(n_0-4)(d_1-24\beta_0)}{n_0-28},& &c_2=c_1\frac{n_0-10}{n_0-28},&\\
&d_3=\frac{(n_0-8)(d_2-24\beta_0)}{n_0-32},& &c_3=c_2\frac{n_0-14}{n_0-32}.&
\end{align*}
The function
$$\varphi_6=(c_1-c_3)E_6\phi+(c_3-1)E_4f_2+(1-c_1)f_6=u+O(q) \in J_{6,L}$$
where 
$$u=(d-504\beta_0)(c_1-c_3)+(d_1+240\beta_0)(c_3-1)+d_3(1-c_1)$$
is a holomorphic Jacobi form of weight $6$ because the potential singular Fourier coefficients $S_1$ and $S_2$ have been cancelled.

In view of the singular weight, $\varphi_6=0$ if $\rank(L)\geq 13$. From Remark \ref{r3.3}, we know that if $F$ exists then $d=2k\geq \rank(L)$ and $\beta_0>0$ when $\rank{L}\geq 6$.  By direct calculations, when $n_0=13$ or $14$, $u\neq 0$, which is impossible.

We next assume that $15\leq\rank(L)\leq 23$.  We construct
$$ g=E_4\phi - f_4 =(d+240\beta_0)-d_2 + (1-c_2)S_2+\cdots \in J_{4,L}^{w.h.}.$$
Since $g$ only has singular Fourier coefficients of the form $S_2$, the minimum possible hyperbolic norm of its Fourier coefficients is $-\frac{1}{2}$. Therefore,  $\eta^6g$ is a holomorphic Jacobi form of weight 7 and index $L$ with character, where $\eta$ is the Dedekind eta function. In view of the singular weight, we have that $\eta^6g=0$ and then $g=0$. If $S_2=0$, then $F$ is a $2$-reflective modular form with complete $2$-divisor, which gives that $L$ is a unimodular lattice of rank $16$ by Theorem \ref{th}. If $S_2\neq 0$, then $1-c_2=0$, which gives $n_0=17$. By $(d+240\beta_0)-d_2=0$ and $n_0=17$, we get $d=150\beta_0$. We hence complete the proof.
\end{proof}

\begin{remark}\label{r1}
Firstly, there exist $2$-reflective lattices when $1\leq \rank(L) \leq 8$. When $\rank(L)=11,12$, we do not know if there exists any $2$-reflective lattice. When $\rank(L)=9,10,17$, there exist $2$-reflective lattices. They are constructed as follows:
\begin{itemize}
\item[(a)] $L=E_8\oplus A_1$:  $\Borch(E_4E_{4,1} \otimes \vartheta_{E_8}/\Delta)$ is a $2$-reflective modular form of weight $195$.
\item[(b)] $L=E_8\oplus 2A_1$:  $\Borch(E_{4,1}\otimes E_{4,1}\otimes\vartheta_{E_8}/\Delta)$ is a $2$-reflective modular form of weight $138$.
\item[(c)] $L=2E_8\oplus A_1$:  $\Borch(E_{4,1}\otimes\vartheta_{E_8}\otimes\vartheta_{E_8}/\Delta)$ is a $2$-reflective modular form of weight $75$.
\end{itemize}
Here the Borcherds products are constructed from weakly holomorphic Jacobi forms of weight $0$ (see \cite{G3} or \cite{GN1v1}). The function $\vartheta_{E_8}$ denotes the theta series for the root lattice $E_8$ which is a holomorphic Jacobi form of weight $4$ and index $E_8$ and the function $E_{4,1}$ denotes the Jacobi-Eisenstein series of weight $4$ and index $1$ introduced in \cite{EZ}. The function $\Delta$ denotes the holomorphic cusp form of weight $12$ on $\SL_2(\ZZ)$.
\end{remark}

We now consider the general case, this means that $M$ does not contain two hyperbolic planes.

\begin{theorem}\label{th2'}
Let $M$ be a $2$-reflective lattice of signature $(2,n)$ with $n\geq 15$. Then $n=19$ or $M$ is isomorphic to the unique even unimodular lattice of signature $(2,18)$ or $(2,26)$.
\end{theorem}

\begin{proof}
The proof is similar to the proof of \cite[Proposition 3.1]{Ma1}. By \cite{Ma1} and \cite[Lemma 1.7]{Ma2}, we know\\
\textbf{1)} If $M$ has a $2$-reflective modular form, then any even overlattice $M'$ of $M$ has a $2$-reflective modular form too. \\
\textbf{2)} One can choose an even overlattice $M'$ of $M$ such that $M'$ contains $2U$.\\
We thus complete the proof by Theorem \ref{th2}.
\end{proof}

\subsection{Application: moduli space of K3 surfaces}
As a first application, we consider the family of lattices
\begin{equation}
T_n=U\oplus U \oplus E_8(-1)\oplus E_8(-1)\oplus \langle -2n \rangle
\end{equation}
where $n\in \NN$. The modular variety $\widetilde{\Orth}^{+}(T_n)\backslash\cD(T_n)$ is the moduli space of polarized K3 surfaces of degree $2n$. The subset
\begin{equation}
Discr=\bigcup_{\substack{l \in T_n\\ (l,l)=-2}} l^{\perp}\cap \cD(T_n)
\end{equation}
is the discriminant of this moduli space.  Nikulin \cite{N} asked the question whether the discriminant is equal to the set of zeros of certain automorphic form. This question is equivalent to whether $T_n$ is 2-reflective. Nikulin showed that for any $N$ there exists $n>N$ such that $T_n$ is not 2-reflective.  Gritsenko and Nikulin \cite{GN2} proved that the lattice $T_n$ is not $2$-reflective if $n>\left(\frac{32}{3}+\sqrt{128+8N}\right)^2$, where $N$ is the integer such that any even integer larger than $N$ can be represented as the sum of the squares of eight different positive integers. Finally, Looijenga \cite{L} proved that $T_n$ is not 2-reflective if $n\geq 2$.  As a direct consequence of Theorem \ref{th1} and Theorem \ref{th2}, we present a pretty simple proof of the result.

\begin{theorem}\label{th3}
The lattice $T_n$ is $2$-reflective if and only if $n=1$.
\end{theorem}

\begin{proof}
By case (c) of Remark \ref{r1}, $T_1$ is $2$-reflective. By contradiction, we assume that $T_n$ is $2$-reflective with $n\geq 2$. On the one hand, by Theorem \ref{th1}, the weight of the corresponding $2$-reflective modular form is 
$$ k=\beta_0\left[ 12+480\left( \frac{12}{17}-\frac{1}{2} \right) \right]=\beta_0\left(110+\frac{14}{17}\right)> 110\beta_0.$$
On the other hand, Theorem \ref{th2} tells that $k=75\beta_0$, which leads to a contradiction. Hence $T_n$ is not $2$-reflective when $n\geq 2$.
\end{proof}

\section{Non-existence of reflective modular forms}\label{Sec:4}
In this section we introduce reflective modular forms and use similar arguments to show the non-existence of reflective modular forms of large rank.

Let $M$ be an even lattice of signature $(2,n)$, $n\geq 3$. The level of $M$ is the smallest positive integer $N$ such that $N(x,x)\in 2\ZZ$ for all $x\in M^\vee$. We remark that any even lattice of squarefree level is of even rank. A primitive vector $l\in M$ of negative norm is called reflective if the reflection 
\begin{equation}
\sigma_l(x)=x-\frac{2(l,x)}{(l,l)}l,  \quad x\in M
\end{equation}
is in $\Orth^+(M)$.

\begin{definition}
A non-constant holomorphic modular form on $\cD(M)$ is called reflective if the support of its divisor is set-theoretically contained in the union of quadratic divisors $l^{\perp}\cap \cD(M)$ determined by reflective vectors $l$ of $M$. A lattice $M$ is called reflective if it admits a reflective modular form.
\end{definition}

A primitive vector $l\in M$ with $(l,l)=-2d$ is reflective if and only if $\div(l)=2d$ or $d$. Let us fix $\lambda=[l/\div(l)]\in A_M$. Then $l^{\perp}\cap \cD(M)$ is contained in $\cH(\lambda,-1/(4d))$ in the first case, and is contained in 
$$\cH(\lambda,-1/d)-\sum_{2\nu=\lambda}\cH(\nu,-1/(4d))$$
in the second case.

Note that all $(-2)$-vectors are reflective. Therefore, $2$-reflective modular forms are particular class of reflective modular forms.

Reflective modular forms of singular weight on lattices of prime level were completely classified by Scheithauer in \cite{S2,S4}. As another application of our arguments in the previous section, we attempt to classify reflective modular forms on lattices of prime level and large rank.

Let $M=2U\oplus L(-1)$ be an even lattice of prime level $p$ and $F$ be a  reflective modular form of weight $k$ with respect to $\widetilde{\Orth}^+(M)$. By \cite[Section 6]{S4}, the divisor of $F$ can be represented as 
\begin{equation}
\div(F)=\beta_0\cH_0 + \sum_{\gamma\in \pi_{M,p}}\beta_\gamma \cH(\gamma, -1/p),
\end{equation}
where $\pi_{M,p}\subset A_M$ is the subset of elements of norm $-2/p$. By \cite{Br2}, there exists a nearly holomorphic modular form with principal part
\begin{equation*}
\beta_0 q^{-1} \textbf{e}_0+ \sum_{\gamma \in \pi_{M,p}} \beta_\gamma q^{-1/p}\textbf{e}_\gamma.
\end{equation*}
Then there exists a weakly holomorphic Jacobi form of weight 0 with singular Fourier coefficients 
\begin{equation}\label{sr}
sing(\psi_L)= \beta_0 \sum_{r\in L}q^{(r,r)/2-1}\zeta^r+ \sum_{\gamma \in \pi_{M,p}} \beta_\gamma\sum_{s\in L +\gamma}q^{(s,s)/2-1/p}\zeta^s.
\end{equation}
Then the $q^0$-term of $\psi_L$ can be written as
\begin{equation*}
\psi_L(\tau,\mathfrak{z})= \beta_0 q^{-1}+\beta_0 \sum_{r\in R(L)}\zeta^r +2k+\sum_{\gamma \in \pi_{M,p}} \beta_\gamma  \sum_{s\in C_\gamma (L)}\zeta^s + O(q),
\end{equation*}
where 
\begin{equation}
C_\gamma (L)=\{ s\in L^\vee: (s,s)=2/p, s-\gamma \in L\}.
\end{equation}
Thus, we get a formula related to the weight of the above reflective modular form
\begin{equation}\label{rwf}
\begin{split}
k=&\beta_0\left[ 12+ \lvert R(L) \rvert \left(\frac{12}{\rank(L)}-\frac{1}{2} \right) \right]\\ &+\left(\frac{12}{p\cdot \rank(L)}-\frac{1}{2} \right)\sum_{\gamma \in \pi_{M,p}} \beta_\gamma\lvert  C_\gamma (L)\rvert.
\end{split}
\end{equation}
It is possible to find a similar formula for the weight of reflective modular forms for general lattices.

\begin{remark}
Let $M=2U\oplus L(-1)$ be an even lattice of prime level $p$ and $F$ be a reflective modular form of weight $k$ for $M$. From (\ref{rwf}), when $\rank(L)=12$ and $p=2$, then $k=\beta_0(12+\frac{1}{2}\lvert R(L) \rvert)$ so the function $F$ is not of singular weight. When $\rank(L)=8$ and $p=3$, $k=\beta_0(12+\lvert R(L) \rvert)$ and $F$ is not of singular weight.
\end{remark}

By \cite[Proposition 3.2]{Ma1}, when a reflective modular form $F$ exists,  we have that either $\rank(L)\leq 23$ or $L$ is a unimodular lattice of rank $24$. We next give a finer classification of reflective modular forms on lattices of prime level.

\begin{theorem}\label{prop}
Let $M=2U\oplus L(-1)$ be an even lattice of prime level $p$. If $M$ admits a reflective modular form of weight $k$ for $\widetilde{\Orth}^+(M)$, then we have
\begin{enumerate}
\item when $p=2$,  either $\rank(L)\leq 16$ or $\rank(L)=20$ and $k=24\beta_0$.
\item when $p=3$,  either $\rank(L)\leq 12$ or $\rank(L)=18$ and $k=48\beta_0$.
\item when $p\geq 5$, $\rank(L)\leq 8+ 24/(p+1)$. 
\end{enumerate}
\end{theorem}

\begin{proof}
Similar to the proof of Theorem \ref{th2}, there exists a weakly holomorphic Jacobi form $\phi$ of weight $0$ and index $L$ with singular Fourier coefficients of the form (\ref{sr}). Assume that $\rank(L)\leq 23$. We write $\phi=S_1 + d + S_2 + \cdots$, where $S_1$ and $S_{2}$ are the first and second terms in (\ref{sr}), respectively, and $d=2k$. We can construct ($n_0=\rank(L)$ and $a=24/p$)
\begin{align*}
f_2&=\frac{24}{n_0-24}H_0(\phi)=S_1+ d_1 +c_1 S_2 +\cdots \in J_{2,L}^{w.h.},\\
f_4&=\frac{24}{n_0-28}H_2(f_2)=S_1+ d_2 +c_2 S_2 +\cdots \in J_{4,L}^{w.h.},\\
f_6&=\frac{24}{n_0-32}H_4(f_4)=S_1+ d_3 +c_3 S_2 +\cdots \in J_{6,L}^{w.h.},
\end{align*}
where
\begin{align*}
&d_1=\frac{n_0(d-24\beta_0)}{n_0-24},& &c_1=\frac{n_0-a}{n_0-24},&\\
&d_2=\frac{(n_0-4)(d_1-24\beta_0)}{n_0-28},& &c_2=c_1\frac{n_0-a-4}{n_0-28},&\\
&d_3=\frac{(n_0-8)(d_2-24\beta_0)}{n_0-32},& &c_3=c_2\frac{n_0-a-8}{n_0-32}.&
\end{align*}
We can check that the function
$$\varphi_6=(c_1-c_3)E_6\phi+(c_3-1)E_4f_2+(1-c_1)f_6=u+O(q) \in J_{6,L}$$
has no term of the form $S_1$ or $S_2$ and then it has no singular Fourier coefficient, so it
is a holomorphic Jacobi form of weight 6, where
$$u=(d-504\beta_0)(c_1-c_3)+(d_1+240\beta_0)(c_3-1)+d_3(1-c_1).$$

We also construct a weakly holomorphic Jacobi form of weight 4
$$ 
g=E_4\phi - f_4 =(d+240\beta_0)-d_2 + (1-c_2)S_2+\cdots \in J_{4,L}^{w.h.}.
$$
By Theorem \ref{th}, we have $S_2\neq 0$ when $n_0> 8$.
By direct calculations, we get
\begin{equation}\label{gg}
c_2=1 \iff \rank(L)= 14+\frac{12}{p}.
\end{equation}
Therefore, when $p=2$, $c_2=1$ if and only if $n_0=20$, when $p=3$, $c_2=1$ if and only if $n_0=18$, when $p>3$, $c_2-1\neq 0$. We thus obtain
\begin{itemize}
\item when $p=2$, if $g=0$ then $n_0=20$ and $d=48\beta_0$;
\item when $p=3$,  if $g=0$ then $n_0=18$ and $d=96\beta_0$;
\item when $p\geq 5$, $g\neq 0$.
\end{itemize}

Suppose $g\neq 0$ and $n_0>8$.  Then $c_2\neq 1$, otherwise $g$ will be a holomorphic Jacobi form of weight $4$, which contradicts the singular weight.
The weakly holomorphic Jacobi form $g$ corresponds to a nearly holomorphic vector-valued modular form $$F=\sum_{\gamma\in A_M}F_\gamma\textbf{e}_\gamma$$
of weight $4-n_0/2$. Note that $F_0(\tau)$ has no term $q^n$ with negative $n$ and for any nonzero $\gamma\in A_M $, the possible term $q^n$ with negative $n$ of $F_\gamma$ is $q^{-1/p}$.
We know from \cite[Proposition 5.3]{Sch15} that $F_0 \neq 0$ because the function $\sum_{\sigma\in \Orth(A_M)} \sigma \cdot F$ is invariant under the orthogonal group $\Orth(A_M)$ of the discriminant group $A_M$ and it is not zero. In addition,   $F_0$ is a nearly holomorphic modular form of weight $4-n_0/2$ with respect to $\Gamma_0(p)$ and its expansion at the cusp $0$ is a linear combination of $F_\gamma$.  As in the proof of \cite[Proposition 6.1]{S4}, the Riemann-Roch theorem applied to $F_0$ gives
$$ -1 \leq p\nu_0(F_0)+\nu_{\infty}(F_0)\leq \left(4-\frac{n_0}{2}\right) \frac{p+1}{12}.$$
This implies 
$$ n_0\leq 8+ \frac{24}{p+1}.$$

It remains to prove that $M$ is not reflective if $n_0=14$ and $p=3$. But in this case, $u\neq 0$ and then $\varphi_6\neq 0$, which gives a contradiction. The proof is completed.
\end{proof}

Note that when $\rank(L)=16$ and $p=2$, we have $u\equiv 0 $. Therefore, our argument cannot determine the weight of the corresponding reflective modular form.

\begin{remark}
\noindent
\begin{enumerate}
\item The rank of an even positive-definite lattice of level $2$ is known to be divisible by $4$ (see \cite{SV}). Thus, by Theorem \ref{prop}, if $2U\oplus L(-1)$ is a reflective lattice of level $2$, then $\rank(L)$ can only be 4, 8, 12, 16 or 20.  We note that there exist reflective lattices of such ranks.  When $L=D_4$, $2D_4$, $E_8\oplus D_4$, $E_8\oplus 2D_4$ or $2E_8\oplus D_4$, the lattice $2U\oplus L(-1)$ admits a reflective modular form.

\item If $L$ is an even positive-definite lattice of level $3$ and of rank $n$ with determinant $\det( L)=\lvert L^\vee/L\rvert=3^r$, then either $r\in \{ 0,n\}$ and $8\lvert n$, or $0< r < n$ and $2r\equiv n\mod 4$. Theorem \ref{prop} says that $2U\oplus L(-1)$ is a reflective lattice of level $3$ only if $\rank(L)=2,4,6,8,10,12$ or 18.
Reflective lattices of such ranks exist.  When $L=A_2$, $2A_2$, $3A_2$, $4A_2$, $E_8\oplus A_2$, $E_8\oplus 2A_2$ or $2E_8\oplus A_2$, the lattice $2U\oplus L(-1)$ admits a reflective modular form. 
\end{enumerate}
\end{remark}

By the above theorem and the weight formula (\ref{rwf}), it is easy to prove the following criterions.

\begin{corollary}
Suppose that the even lattice $M=2U\oplus L(-1)$ of prime level $p$ is reflective. 
\begin{enumerate}
\item When $\rank(L)=20$ and $p=2$, we have $\lvert R(L) \rvert \geq 120$. 
\item When $\rank(L)=18$ and $p=3$, we have $\lvert R(L) \rvert \geq 216$. 
\end{enumerate}
\end{corollary}

The above result can be used to decide whether a given lattice is reflective or not. For instance, we see at once that $2U\oplus L(-1)$ is not reflective when $L=E_8(2)\oplus 3D_4$ or $2E_6\oplus 3A_2$.

We next extend the above classification results to the general case. The following lemma introduced in \cite[Corollary 3.2]{Ma2} is very useful for our purpose.

\begin{lemma}\label{lemr}
Let $M$ be a lattice of signature $(2, n)$ with $n \geq 11$. There exists
a lattice $M_1$ on $M\otimes \QQ$ such that $\Orth^+(M)\subset \Orth^+(M_1)$ and that $M_1$ is a scaling of an even lattice containing $2U$.
\end{lemma}

We remark that the above lattice $M_1$ is usually not an overlattice of $M$ and it is constructed as a sublattice of an overlattice of $M$. We next use the above lemma to extend Ma's result \cite[Proposition 3.2]{Ma1} to the general case.

\begin{theorem}\label{thma}
\noindent
\begin{enumerate}
\item There is no reflective lattice of signature $(2,n)$ with $n>26$.
\item Let $M$ be an even lattice of signature $(2,26)$. If it admits a reflective modular form which can be constructed as a Borcherds product, then it is isomorphic to the unique even unimodular lattice of signature $(2,26)$.
\end{enumerate}
\end{theorem}

\begin{proof}
We first prove the statement (1). By contradiction, assume that there is a reflective lattice $M$ of signature $(2,n)$ with $n>26$ and we denote the corresponding reflective modular form by $F$. By Lemma \ref{lemr}, 
there exists a lattice $M_1$ on $M\otimes \QQ$ such that $\Orth^+(M)\subset \Orth^+(M_1)$ and that $M_1$ is a scaling of an even lattice $M_2$ containing $2U$. Here, we have a natural isomorphism 
$$
\cD(M) \cong  \cD(M_1) \cong  \cD(M_2),
$$
where the first comes from the equality $M\otimes \QQ=M_1\otimes \QQ$ and the second from the identification $M_1=M_2$ as $\ZZ$-modules. Moreover, the inclusion $\Orth^+(M)\subset \Orth^+(M_1) \cong \Orth^+(M_2)$ is compatible with this isomorphism and the isomorphism preserves the reflective divisors. Thus, $F$ is a reflective modular form for $M_1$ and then also a reflective modular form for $M_2$, which contradicts \cite[Proposition 3.2]{Ma1}. 

We next prove the statement (2). The proof is similar to  \cite[Proposition 3.2]{Ma1}. Assume that the corresponding reflective modular form $F$ is a Borcherds product of a nearly holomorphic modular form $f$. Then $\Delta f$ is a holomorphic modular form of weight $0$ and hence must be an $\Mp_2(\ZZ)$-invariant vector in $\CC[A_M]$. Then we get $M=II_{2,26}$ because $\Delta f$ does not transform correctly under the matrix $S$ when $\lvert A_M \rvert \neq 1$.  This completes the proof.
\end{proof}

\begin{remark}
We do not know if there is any other reflective lattice of signature $(2,26)$ except the scalings of $II_{2,26}$. By the second statement in Theorem \ref{thma} and \cite{Br1}, such reflective lattice is not of the form $U\oplus U(m)\oplus L(-1)$. This question is related to the general question if all reflective modular forms come from Borcherds products.
\end{remark}

We now extend Theorem \ref{prop} to the general case.

\begin{theorem}\label{coro}
Let $M$ be a reflective lattice of signature $(2,n)$ and of prime level $p$.
\begin{enumerate}
\item If $p=2$ and $n>18$ and $n\neq 22$, then $M$ is isomorphic to $II_{2,26}(2)$.
\item If $p=3$ and $n>14$ and $n\neq 20$, then $M$ is isomorphic to $II_{2,18}(3)$ or $II_{2,26}(3)$.
\item If $p>3$ and $n> 10+ 24/(p+1)$, then $M$ is isomorphic to $II_{2,18}(p)$ or $II_{2,26}(p)$.
\end{enumerate}
\end{theorem}

\begin{proof}
Let $F$ be a reflective modular form for $M$. In this case, we have $n\geq 11$. Assume that the determinant of $M$ is $p^a$, where $1\leq a\leq n+2$ is an integer. Since $M$ is of prime level $p$, the discriminant group of $M$ is isomorphic to $(\ZZ/p\ZZ)^a$ and the minimum number of generators of this group is $l(M)=a$. 

If $n>a+2$, by \cite[Corollary 1.10.2]{Nik80}, there exists an even lattice $L$ of sinature $(0,n-2)$ having discriminant form $M^\vee/M$ because $n-2>l(M)$. Then the lattice $2U\oplus L$ is an even lattice of signatute $(2,n)$ with the same discriminant form as $M$. From \cite[Corollary 1.13.3]{Nik80}, it follows that $M$ and $2U\oplus L$ are isomorphic. Thus, the function $F$ is also a reflective modular form for $2U\oplus L$. We then prove this theorem by Theorem \ref{prop}.

If $n\leq a+2$, then $a\geq 9$ because $n\geq 11$. Since $M$ is of prime level $p$, the lattice $M^\vee(p)$, which is a scaling of the dual lattice of $M$, is even and of determinant $p^{2+n-a}$. Moreover, if $M^\vee(p)$ is not unimodular, then it is of level $p$. In view of the natural isomorphism 
$$\Orth^+(M)\cong \Orth^+(M^\vee)\cong \Orth^+(M^\vee(p)),$$
the function $F$ is a reflective modular form for $M^\vee(p)$. Since $n>(2+n-a)+2$, we can prove this case as the previous case. If $M^\vee(p)$ is  unimodular, then $M=(M^\vee(p))^\vee(p)$ is a scaling of an even unimodular lattice. Thus, the proof is completed.
\end{proof}

\begin{remark}
If a reflective modular form can be constructed as a Borcherds product of a vector-valued modular form invariant under $\Orth(A_M)$, it is called symmetric, otherwise it is called non-symmetric.
\cite[Theorem 6.5]{S4} gives bounds on the signature for non-symmetric reflective modular forms. But the bounds do not hold in the symmetric case. However, the above result gives bounds in the symmetric case.
\end{remark}

\begin{theorem}\label{prop01}
There is no reflective lattice of signature $(2,n)$ with $23\leq n\leq 25$.
\end{theorem}

\begin{proof}
Let $M$ be an even lattice of signature $(2,n)$ with $23\leq n\leq 25$.
Firstly, we assume that $M$ contains $2U$. By contradiction, assume that $M$ is reflective. Then, there exists a weakly holomorphic Jacobi form $\phi$ of weight 0 and we can construct a weakly holomorphic Jacobi form $f_4$ of weight 4 from $\phi$ as in the proof of Theorem \ref{prop}. We define $g=E_4\phi - f_4$. Then $g$ has no singular Fourier coefficient of hyperbolic norm $-2$. From (\ref{gg}), we get $g\neq 0$. Moreover, the minimum possible hyperbolic norm of the Fourier coefficients of $g$ is $-1$. Then $\eta^{12}g$ is a holomorphic Jacobi form of weight $10$ with character, which leads to a contradiction due to the singular weight. The general case can be proved as the proof of Theorem \ref{thma}.
\end{proof}

\begin{remark}
When $1\leq \rank(L) \leq 20$ and $\rank(L)\neq 15$ or $19$, there exist reflective lattices $2U\oplus L(-1)$, such as $A_n$ for $1\leq n \leq 7$, $D_8$, $E_8\oplus A_1$, $E_8\oplus A_2$, $E_8\oplus A_2\oplus A_1$, $E_8\oplus D_4$, $E_8\oplus D_4\oplus A_1$, $E_8\oplus D_4\oplus A_2$, $E_8\oplus 2D_4$, $2E_8\oplus A_1$, $2E_8\oplus A_2$, $2E_8\oplus D_4$. But we do not know if there exists reflective lattice $2U\oplus L(-1)$ with $\rank(L)=15$ or $19$.
\end{remark}

\begin{questions} Here, we would like to formulate some interesting questions related to our work:
\begin{enumerate}
\item Are there 2-reflective lattices of signature $(2,13)$ or $(2,14)$?
\item Are there reflective lattices of signature $(2,17)$ or $(2,21)$?
\item Let $M=2U\oplus L(-1)$ be a 2-reflective lattice of signature $(2,19)$. Is $L$ equal to $2E_8\oplus A_1$ up to isomorphism? 
\item Classify the following interesting lattices:
\begin{itemize}
\item 2-reflective lattices of signature $(2,19)$,
\item reflective lattices of signature $(2,22)$ and of level 2,
\item reflective lattices of signature $(2,20)$ and of level 3.
\end{itemize}  
\end{enumerate}
\end{questions}

\section{Another application: dd-modular forms}\label{Sec:5}
Our arguments in the previous section are also applicable to some other questions. In this section we use similar arguments to classify the modular forms with the simplest reflective divisors, i.e. the dd-modular forms defined in \cite{CG1}.

Let $nA_1$ denote the lattice of $n$ copies of $A_1=\langle 2 \rangle$, $n\in \NN$. Let $\{e_1, ...,e_n\}$ denote the standard basis of $\RR^n$ with standard scalar product $(\cdot,\cdot)$. We choose the following model for the lattice $nA_1(m)$:
\begin{equation}
(\langle e_1,...,e_n \rangle_\ZZ,\; 2m(\cdot ,\cdot) )
\end{equation}
and set $\mathfrak{z}_n=\sum_{i=1}^n z_ie_i \in nA_1 \otimes \CC$, $\zeta_i=e^{2\pi i z_i}$, for $1\leq i \leq n$. We define
\begin{equation}
\Gamma_{n,m}=\Orth^{+}(2U\oplus nA_1(-m))
\end{equation}
and the definition of dd-modular forms is as follows.
\begin{definition}
A holomorphic modular form with respect to $\Gamma_{n,m}$ is called a dd-modular form if it vanishes exactly along the $\Gamma_{n,m}$-orbit of the diagonal $\{z_n=0\}$. The $\Gamma_{n,m}$-orbit of the diagonal $\{z_n=0\}$, denoted by $\Gamma_{n,m}\{z_n=0\}$, is called the diagonal divisor.
\end{definition}

It is well known that the Igusa form $\Delta_5$ which is the product of the ten even theta constants vanishes precisely along the diagonal divisor $\{z=0\}$. Therefore, the dd-modular form is a natural generalization of $\Delta_5$.  Gritsenko and Hulek \cite{GH} proved that the dd-modular form exists for the lattice $A_1(m)$ if and only if $1\leq m \leq 4$.  Cl\'ery and Gritsenko \cite{CG1} developed the arguments in \cite{GH} and gave the full classification of the dd-modular forms with respect to the Hecke subgroups of the Siegel paramodular groups. But their approach is hard to generalize to higher dimensions. Since dd-modular forms are crucial in determining the structure of the fixed space of modular forms and have applications in physics, as an important application of our arguments, we prove the following classification results for all dd-modular forms for lattices of the shape $nA_1$.

\begin{theorem}\label{th4}
The dd-modular form exists if and only if the pair $(n,m)$ takes one of the eight values
\begin{align*}
&(1,1),& &(1,2),& &(1,3),& &(1,4),& &(2,1),& &(2,2),& &(3,1),& &(4,1).&
\end{align*}
\end{theorem}
\begin{proof}
Suppose that $F_c$ is a modular form of weight $k$ with respect to $\Gamma_{n,m}$ with the divisor  $c\cdot\Gamma_{n,m}\{z_n=0\}$, where $c$ is the multiplicity of the diagonal divisor and it is a positive integer. The diagonal divisor $\Gamma_{n,m}\{z_n=0\}$ is the union of the primitive Heegner divisors $\cP(\pm e_i/(2m),-1/(4m))$, $1\leq i \leq n$, where the primitive Heegner divisor of discriminant $(\mu, d)$ is defined as
$$ \cP(\mu, d)= \bigcup_{\substack{ M+\mu \ni l\, \textit{primitive}  \\ (l,l)=2d}} l^{\perp}\cap \cD(M).$$ 
It is clear that we have
$$ \cP\left(\mu,y\right) = \cH \left(\mu,y\right) -  \sum_{d > y} x_d \cH (\lambda_d,d),  $$
where $x_d$ are integers and $\lambda_d\in M^\vee$ (we refer to \cite[Lemma 4.2]{BrM} for an explicit formula).
For arbitrary Heegner divisor $\cH(\lambda,d)$ with $\lambda=(0,n_1,\lambda_0,n_2,0) \in A_M$, the principal part of the corresponding nearly holomorphic modular form of weight $-\rank(L)/2$ with respect to the Weil representation $\rho_M$ of $\Mp_2(\ZZ)$ is $q^d\textbf{e}_\lambda$. Hence the singular Fourier coefficients of the corresponding weakly holomorphic Jacobi form of weight $0$ are represented as
$$ \sum_{r\in L+\lambda_0}q^{(r,r)/2+d}e^{2\pi i(r,\mathfrak{z})}. $$
Since $\{\pm e_i/(2m):1\leq i\leq n\}$ is the set of vectors in $nA_1(m)^\vee$ with the minimum norm $1/(2m)$ in $nA_1(m)^\vee/nA_1(m)$, through the previous explanations, there exists a weak Jacobi form $f_{nA_1,m}$ of weight $0$ and index $nA_1(m)$ satisfying
$$ f_{nA_1,m}=c \cdot \sum_{1\leq i \leq n}\zeta_i^{\pm 1}+2k +O(q)$$
such that $F_c$ is the Borcherds product of $f_{nA_1,m}$.
We next apply Lemma \ref{L2} to our case. In the case, $a=0$, $\rank(L)=n$. For each term $c\zeta_i$ or $c\zeta_i^{-1}$, the corresponding $l=\pm\frac{1}{2m}e_i$, $c(0,l)(l,l)=c\cdot 2m \cdot \frac{1}{4m^2}=\frac{c}{2m}$. Therefore, we get 
\begin{equation*}
m(2nc+2k)=12c,
\end{equation*}
then $nm\leq 5$. It is not hard to show that a weak Jacobi form for $nA_1(m)$ has integral Fourier coefficients if its $q^0$-term is integral when $nm\leq 5$.

When $m\leq 4$, $2k$ is integral if $c=1$. Hence the existence of $F_c$ is equivalent to the existence of $F_1$.
In view of $k\geq n/2$, then the triplet $(m,n,k)$ can only take one of the eight values
$$ (1,1,5), (1,2,4), (1,3,3), (1,4,2), (2,1,2), (2,2,1), (3,1,1), (4,1,\frac{1}{2}). $$

When $m=5$, we only need to consider the case of $c=5$, and we obtain the unique solution $(5,1,1)$. But the unique weak Jacobi form of weight $0$ and index $5$ for $A_1$ is $\psi_{0,5}^{(1)}=5\zeta^{\pm 1}+2+q(-\zeta^5+\cdots)$ (see \cite[formula (1.12)]{G1}). The corresponding Borcherds product is not holomorphic, that is, $F_c$ does not exist in the case. We have thus proved the theorem. 
\end{proof}

\begin{remark}
Similarly, we can define dd-modular forms with respect to the lattices $A_n(m)$ or $D_n(m)$. Using the same methodology, we can easily classify these dd-modular forms. In fact, dd-modular forms with respect to the lattices $L(m)$, where $L=A_n, n\geq 2$ or $L=D_n, n\geq 4$, exist if and only if the pair $(L,m)$ takes one of the following fifteen values
\begin{align*}
&(A_2,1)& &(A_3,1)& &(A_4,1)& &(A_5,1)& &(A_6,1)& &(A_7,1)& &(A_2,2)& &(A_2,3)&\\ &(A_3,2)& &(D_4,1)& &(D_5,1)& &(D_6,1)& &(D_7,1)& &(D_8,1)& &(D_4,2).& 
\end{align*}
Note that all dd-modular forms in Theorem \ref{th4} and in the above list do exist and can be found in \cite{CG1, G4, GN1v1,GN3}.
\end{remark}

\bigskip

\noindent
\textbf{Acknowledgements} The author would like to thank his supervisor Valery Gritsenko for helpful discussions and constant encouragement. The author thanks Shouhei Ma for explaining the proof of Lemma \ref{lemr} to him.  The author also thanks the anonymous referees for their careful reading and useful suggestions. This work was supported by the Labex CEMPI (ANR-11-LABX-0007-01) in the University of Lille.

\bibliographystyle{amsplain}

\end{document}